\documentclass{amsart}
\usepackage{amsmath,amsthm,amssymb}

\makeatletter
    
    \@addtoreset{equation}{section}
  \makeatother

\newtheorem{definition}{Definition}[section]
\newtheorem{proposition}[definition]{Proposition}
\newtheorem{theorem}[definition]{Theorem}
\newtheorem{lemma}[definition]{Lemma}
\newtheorem{corollary}[definition]{Corollary}
\newtheorem{remark}{Remark}[section]

\newcommand{\re}{\mathbb R} 

\def \pt{\partial}

\DeclareMathOperator{\diver}{div}

\pagebreak

\title[3D axisymmetric Navier-Stokes equations]{
Asymptotic properties of steady solutions \\
to the 3D axisymmetric Navier-Stokes equations \\
 with no swirl
}
\author{Hideo Kozono, Yutaka Terasawa and Yuta Wakasugi}

\address[H. Kozono]{Department of Mathematics, Faculty of Science and Engineering,
Waseda University, Tokyo 169--8555, Japan, 
Research Alliance Center of Mathematical Sciences, Tohoku University, 
Sendai 980-8578, Japan}
\email[H. Kozono]{kozono@waseda.jp, hideokozono@tohoku.ac.jp}
\address[Y. Terasawa]{Graduate School of Mathematics, Nagoya University,
Furocho Chikusaku Nagoya 464-8602, Japan}
\email[Y. Terasawa]{yutaka@math.nagoya-u.ac.jp}
\address[Y. Wakasugi]{Graduate School of Science and Engineering,
Hiroshima University,
Higashi-Hiroshima, 739-8527, Japan}
\email[Y. Wakasugi]{wakasugi@hiroshima-u.ac.jp}
\begin{document}
\begin{abstract}
We study the asymptotic behavior of axisymmetric solutions with no swirl to
the steady Navier-Stokes equations in the outside of the cylinder.
We prove an a priori decay estimate of the vorticity
under the assumption that the velocity has generalized finite Dirichlet integral.
As an application, we obtain a Liouville-type theorem.
\end{abstract}

\maketitle
\section{Introduction}
\footnote[0]{2010 Mathematics Subject Classification. 35Q30; 35B53; 76D05}

We study the asymptotic behavior of axisymmetric solutions with no swirl to
the steady Navier-Stokes equations
\begin{align}
\label{sns}
	\left\{ \begin{array}{l}
		(v \cdot \nabla )v + \nabla p = \Delta v, \\
		\nabla \cdot v = 0,
		\end{array} \right.
		\quad x \in D,
\end{align}
where $D$ is the outside of cylinders in $\re^3$ specified later,
and
$v = v(x) = (v_1(x), v_2(x), v_3(x))$
and
$p = p(x)$
denote the velocity vector field and the scalar pressure
at the point $x = (x_1, x_2, x_3) \in D$, respectively.
\par
To state previous results, 
we temporary append the condition at infinity
\begin{align}%
\label{vinf}
	\lim_{|x|\to \infty} v(x) = v_{\infty}
\end{align}%
with a given constant vector $v_{\infty}$.
For the stationary Navier-Stokes equations \eqref{sns}--\eqref{vinf}
in general exterior domains $D$
with the condition $v=0$ at the compact boundary $\partial D$ and
a given smooth external force $f$,
Leray \cite{Le} proved that there exists at least one smooth solution
$(v,p)$.
The solution constructed in \cite{Le} has the finite Dirichlet integral.
In general, the solution $v$ of (\ref{sns}) having the bounded Dirichlet integral, 
i,e., $\int_D|\nabla v(x)|^2 dx < \infty$ is called
a D-solution.
Although the convergence \eqref{vinf} had been shown in such a  weak sense 
as $\int_D|v(x) - v_\infty|^6dx < \infty$,  
later on,
Finn \cite{Fi59} proved that any D-solution converges to the prescribed constant vector uniformly at infinity.
After that, Finn \cite{Fi59b} introduced the notion of the physically reasonable (PR) solution $v$ of (\ref{sns}) which satisfies
$v(x) = O(|x|^{-1})$ if $v_{\infty} = 0$ and 
$v(x) - v_{\infty} = O(|x|^{-\frac12-\varepsilon})$ if $v_{\infty} \neq 0$
with some $\varepsilon > 0$.
Then, Finn \cite{Fi65} constructed a PR-solution, provided that the data are sufficiently small.
Furthermore, in the case where $v_{\infty} = 0$ and the external force is sufficiently small,
Galdi--Simader \cite{GaSi}, Novotny--Padula \cite{NoPa},
and Borchers--Miyakawa \cite{BoMi}
constructed a solution satisfying
$v(x) = O(|x|^{-1})$ and $\nabla v(x) = O(|x|^{-2})$.
We note that, in particular, this solution satisfies 
$\nabla v \in L^q$ for all $q > 3/2$
(we refer \cite{KoYa98} by the first author and Yamazaki
for construction of the solution in the class $\nabla v \in L^{3/2, \infty}$).
\par
It has been an important problem to study the relation between the D-solution and the PR-solution.
It is easily proved that every PR-solution is necessarily a D-solution, however,
the converse implication, 
namely, the precise asymptotic behavior of D-solutions,
had been an open question.
For that question,
Babenko \cite{Ba73} proved that
if $v_{\infty} \neq 0$ and the external force $f$ is compactly supported, then
any D-solution is a PR-solution. On the other hand, in the case of $v_{\infty} = 0$, much less is known.
Galdi \cite{Ga92} showed the same result as \cite{Ba73} when $v_{\infty} = 0$,
provided that the data are sufficiently small.
\par
In order to study further the asymptotic behavior of solutions 
when $v_{\infty} = 0$,
recently, axisymmetric solutions are fully investigated.
For the axisymmetric solutions, we may expect that the
situation becomes similar to that of the 2-dimensional case 
in which the asymptotic behavior of the solution is well-studied.
For the literature of 2-dimensional problems, we refer the reader to
\cite{GiWe78, Am88, KoPiRu19, KoPiRu20}.
\par
\bigskip
In what follows,
we use the cylindrical coordinates
$r = \sqrt{x_1^2 + x_2^2}$,
$\theta = \tan^{-1} (x_2/x_1)$,
$z = x_3$,
and let
$e_r = (x_1/r, x_2/r, 0)$,
$e_{\theta} = (-x_2/r, x_1/r, 0)$,
and
$e_z = (0,0,1)$.
Using such $\{e_r, e_{\theta}, e_z\}$ as an orthogonal basis in $\re^3$, 
we express the vector field $v = v(x)$ as
\begin{align*}%
	v(x) = v^r(r,\theta,z) e_r + v^{\theta}(r,\theta, z) e_{\theta} + v^z(r,\theta,z) e_z.
\end{align*}%
By axisymmetry we mean that $v_r$, $v_\theta$ and $v_z$ are independent of 
$\theta$.  
Choe--Jin \cite{ChoJi}, Weng \cite{We}, and Carrillo--Pan--Zhang \cite{CaPaZh}
showed that an axisymmetric solution of \eqref{sns}--\eqref{vinf} in $D = \re^3$ 
with  $v_{\infty} = 0$ in the class the finite Dirichlet integral
satisfies
\begin{eqnarray}%
	&& v(x) = O\left( \frac{(\log r)^{\frac12}}{r^{\frac12}} \right),  \nonumber \\
	&&
	| \omega^{\theta} | = O \left( \frac{(\log r)^{\frac34}}{r^{\frac54}} \right),\quad
	| \omega^r | + |\omega^z| = O \left( \frac{(\log r)^{\frac{11}{8}}}{r^{\frac98}} \right) 
	\label{est_CaPaZh}
\end{eqnarray}%
uniformly in $z$ as $r \to \infty$,
where $\omega^{r}, \omega^{\theta}, \omega^z$ are the components of the vorticity
$\omega$
defined by
\begin{equation}\label{eqn:1.4}%
	\omega^{r} = - \partial_z v^{\theta},\quad
	\omega^{\theta} = \partial_z v^r - \partial_r v^z,\quad 
	\omega^z = \frac{1}{r} \partial_r (r v^{\theta}).
\end{equation}%

Recently, Li--Pan \cite{LiPa} studied 
a similar asymptotic behavior of solutions in the class of  
the finite generalized Dirichlet integral
\begin{align}
\label{Diri}
	\int_{\mathbb{R}^3} |\nabla v(x) |^q \,dx < + \infty
\end{align}
for some $q \in (2,\infty)$.
They first showed
\begin{align}%
\label{lipa_v}
	\begin{cases}
		\exists v_{\infty}^z \in \mathbb{R}, \ | v(r,z) - (0,0,v_{\infty}^z) | = O(r^{1-\frac3q}),
			&q \in (2,3),\\
		\mbox{for any} \ r_0 > 0 \ \mbox{and}\ r > r_0, \ | v(r,z) - v(r_0,z) | = O \left( \log \frac{r}{r_0} \right),	
			&q = 3,\\
		\mbox{for any} \ r_0 \ge 0 \ \mbox{and} \ r > r_0, \ |v(r,z) - v(r_0,z) | = O(r^{1-\frac3q}),
			&q \in (3,\infty)
	\end{cases} 
\end{align}%
as $r \to \infty$ uniformly in $z \in \mathbb{R}$.
Then, for the behavior of $\omega$, they obtained 
\begin{align}%
\label{lipa_om1}
		|\omega^{\theta}(r,z)| = O(r^{-(\frac1q + \frac{3}{q^2} + \varepsilon)}), \quad
		| \omega^r (r,z) | + |\omega^z(r,z) | 
		= O(r^{-(\frac1q + \frac{1}{q^2} + \frac{3}{q^3} + \varepsilon)} ),
\end{align}%
for an arbitrary $\varepsilon>0$  
provided that
$q \in [3,\infty)$ and that $\sup_{z\in \mathbb{R}} |u(r_0,z)| \le C$ for some $r_0 > 0$ hold,
or, $q \in (2,3)$ and $v_{\infty}^z = 0$ hold.
Besides them, they also showed
\begin{align}%
\label{lipa_om2}
	|\omega^{\theta}(r,z)| = O(r^{-\frac2q + \varepsilon}), \quad
| \omega^r (r,z) | + |\omega^z(r,z) | = O(r^{-(\frac1q + \frac{2}{q^2}) + \varepsilon} ),
\end{align}%
for an arbitrary  $\varepsilon > 0$
provided that $q \in (2,3)$ and that $v_{\infty}^z \neq 0$ hold.
\par
\vspace{2mm}
In this paper,
we treat general axisymmetric exterior domains, and
we prove that if the velocity has no swirl, that is, $v^{\theta} \equiv 0$,
then the decay rates of $\omega^{\theta}$
obtained in \cite{CaPaZh} and \cite{LiPa} are further improved.
We note that, as described before, solutions satisfying \eqref{Diri} are constructed in
\cite{GaSi}, \cite{NoPa}, \cite{BoMi}, \cite{KoYa98},
and concerning the decay property at infinity, 
the condition \eqref{Diri} with $q > 2$
is weaker than that in the class of the finite Dirichlet integral, namely 
with $q =2$.
Furthermore, as a byproduct of our result, we show a Liouville-type theorem.
Although our result needs an additional assumption that
$\lim_{|z| \to \infty} \omega^{\theta}(r,z) = 0$,
we may treat the case when the velocity grows at infinity.
For other Liouville-type theorems,
we refer the reader to
\cite{KoNaSeSv09}, \cite{KoPiRu15}, \cite{LeZhZh17}, \cite{Zh}, and \cite{Wa} for axisymmetric solutions,
and \cite{Ga}, \cite{Ch1}, \cite{Ch15}, \cite{ChWo}, \cite{Se}, \cite{KoTeWa17} for general cases, respectively.  
\par
\bigskip
To state our main results, we introduce some notations.
We consider the cylindrical domain
$D = \{ (r,\theta, z) \in \re_+ \times [0,2\pi) \times \re;\mbox{ } r > r_0\}$
with some constant $r_0 > 0$,
and let
$D_0 = \{ (r,z) \in \re_+ \times \re;\mbox{ } r > r_0 \}$.
By the axisymmetric velocity with no swirl, we mean that
$v^r$ and $v^z$ are independent of $\theta$ and $v^{\theta} \equiv 0$.
From this, we rewrite the equation \eqref{sns} as
\begin{align}
\label{snsaxi}
	\left\{ \begin{array}{l}
		\displaystyle (v^r \pt_r + v^z \pt_z ) v^r + \pt_r p
			= \left( \pt_r^2 + \frac{1}{r}\pt_r + \pt_z^2 - \frac{1}{r^2} \right) v^r,\\
		\displaystyle (v^r \pt_r + v^z \pt_z ) v^z + \pt_z p
			= \left( \pt_r^2 + \frac{1}{r}\pt_r + \pt_z^2 \right) v^z,\\
		\displaystyle \pt_r v^r + \frac{v^r}{r} + \pt_z v^z = 0,
	\end{array} \right.
	\quad (r, z) \in D_0.
\end{align}
Since $v^\theta =0$, it follows from (\ref{eqn:1.4}) that 
$\omega^r= \omega^z = 0$, and we see that 
$\omega^{\theta} = \pt_z v^r - \pt_r v^z$ satisfies the vorticity equation
\begin{align}
\label{eq_w}
	(v^r \pt_r + v^z \pt_z ) \omega^{\theta} - \frac{v^r}{r} \omega^{\theta}
	 = \left( \pt_r^2 + \frac{1}{r} + \pt_z^2 - \frac{1}{r^2} \right) \omega^{\theta}.
\end{align}
Moreover,
$\Omega \equiv \dfrac{\omega^{\theta}}{r}$ is subject to the identity 
\begin{align}
\label{eq_omega}
	- \left( \pt_r^2 + \pt_z^2 + \frac{3}{r} \pt_r \right) \Omega
	 	+ \left( v^r \pt_r + v^z \pt_z \right) \Omega = 0.
\end{align}
The equation \eqref{eq_omega} has a similar structure to
the vorticity equation of the 2-dimensional Naver-Stokes equations.
In particular, every solution $\Omega$ to \eqref{eq_omega} satisfies the maximum principle, which means that for each bounded subdomain $D_1 \subset D_0$, 
 $\Omega$ attains its maximum and minimum on the boundary of $D_1$.    
In addition to the assumption \eqref{Diri}, we suppose that
there exist $k \in \re $ and $C>0$ such that
\begin{align}
\label{ass_v}
	|v(r,z)| \le C (1+|r|)^k
\end{align}
holds for all $(r,z) \in D_0$.
\par
\vspace{2mm}
Under the assumptions \eqref{Diri} and \eqref{ass_v},
we have the following asymptotic behavior of the vorticity
$\Omega$ and $\omega^{\theta}$.
\begin{theorem}\label{thm_asymp}
Let $(v,p)$ be a smooth axisymmetric solution of \eqref{sns} with no swirl
satisfying $\displaystyle{\int_D|\nabla v(x)|^qdx < \infty}$ with some $q \in [2,\infty)$. Assume that $v$ fulfills \eqref{ass_v}.
Then, we have
\begin{align}
\label{asym_omega}
	&\lim_{r\to \infty} r^{ 1+\frac{3}{q} - \frac{1}{2q} \max\{0,1+k\} } \sup_{z \in \re} |\Omega(r,z)| = 0,\\
\label{asym_omega_theta}
	&\lim_{r\to \infty} r^{ \frac{3}{q} - \frac{1}{2q} \max\{0,1+k\} }
		\sup_{z \in \re} |\omega^{\theta} (r,z)| = 0
\end{align}
as $r\to \infty$.
\end{theorem}

\begin{remark}
{\rm 
We note that the asymptotic behavior \eqref{lipa_v} obtained by \cite{LiPa} is also true
in the outside $D$ of cylinder. 
Thus, we can choose $k$ in the assumption \eqref{ass_v} from \eqref{lipa_v}.
For example, when
$q \in (2,3)$ and $v_{\infty}^z \neq 0$, the assumption \eqref{ass_v} is valid with $k = 0$,
and when $q \in (2,3)$ and $v_{\infty}^z = 0$, the assumption \eqref{ass_v} is valid with $k = 1-\frac3q$.
}
\end{remark}

\begin{remark}
{\rm (i)
Compared with the estimates of $\omega^{\theta}$ in \eqref{est_CaPaZh},
we note that the above theorem in the case where $D=\re^3$ and $q = 2$ gives a better decay.
Indeed, we already know
$|v(x)| = O( r^{-\frac12} \sqrt{\log r})$
from the first estimate of \eqref{est_CaPaZh}.
Thus, $v$ satisfies the assumption \eqref{ass_v} with $k = - \frac12 + \varepsilon$ for arbitrary  small $\varepsilon > 0$.
Then, the estimate \eqref{asym_omega_theta} imples
$|\omega^{\theta}(r,z)| = o ( r^{-\frac{11}{8}+ \frac{\varepsilon}{4}} )$,
which is slightly better than that of \eqref{est_CaPaZh}.
As we pointed out, in the case of no swirl, both $\omega^r$ and $\omega^z$ vanish,
and we emphasize that our result is valid
not only in $\mathbb{R}^3$ but also in exterior domains of cylinders.
\par
(ii)
Let us compare our result with that of \cite{LiPa}.
We first note that the estimates of $\omega^{\theta}$ in \eqref{lipa_om1} and \eqref{lipa_om2}
are also true in the outside $D$ of the cylinder, 
while the estimates of $\omega_r$ and $\omega_z$ are proved only in $\mathbb{R}^3$.
In comparison with the estimates \eqref{lipa_om1} and \eqref{lipa_om2} for $q \in (2,\infty)$,
we note that Theorem \ref{thm_asymp} gives slightly better estimates.
Indeed, for example, when $q \in (2,3)$ and $v_{\infty}^z \neq 0$ in \eqref{lipa_v},
the assumption \eqref{ass_v} is valid for $k = 0$.
Then, Theorem \ref{thm_asymp} implies
$|\omega^{\theta}(r,z)| = o(r^{-\frac{5}{2q}} )$,
while \eqref{lipa_om2} shows $|\omega^{\theta}(r,z)| = O(r^{-\frac2q + \varepsilon})$.
The other cases are also discussed similarly.
}
\end{remark}
\par
As a byproduct of Theorem \ref{thm_asymp},
by the maximum principle to $\Omega$,
we immediately have the following Liouville-type theorem.
\begin{corollary}\label{cor_liouville}
Let $D = \re^3$ and
let $(v,p)$ be a smooth axisymmetric solution of \eqref{sns} with no swirl
having the finite generalized Dirichlet integral as in (\ref{Diri})
for  some $q \in [2,\infty)$ with the condition 
 \eqref{ass_v} for some $k \le 2q + 5$.
Moreover, for every $r >0$  we assume that 
$\lim_{|z|\to \infty} |\omega^{\theta} (r,z)| = 0$. 
Then, we have
$\omega^{\theta} \equiv 0$, and hence $v$ is harmonic on $\re^3$.
\end{corollary}

\section{Proof of Theorem \ref{thm_asymp}}
In what follows, we denote by $C$ generic constants.
In particular, $C = C(*,\ldots,*)$
denotes constants depending only on the quantities appearing in the parenthesis.
We sometimes use the symbols
$\diver_{r,z}, \nabla_{r,z}$, and $\Delta_{r,z}$,
which mean the differential operators defined by
$\diver_{r,z} (f_1, f_2) (r,z) = \partial_r f_1(r,z) + \partial_z f_2(r,z)$,
$\nabla_{r,z} f(r,z) = (\partial_r f, \partial_z f)(r,z)$,
and
$\Delta_{r,z} f(r,z) = \partial_r^2 f (r,z) + \partial_z^2 f(r,z)$,
respectively.

\subsection{$L^q$-energy estimates}

Since
\begin{align*}
	|\nabla_x v(x)|^2 = | \nabla_{r,z} v |^2 + \frac{1}{r^2} |v^r(r,z)|^2
\end{align*}
for the axisymmetric vector field $v$ without swirl,
we first note that the condition \eqref{Diri} implies
\begin{align*}
	\infty > \int_D |\nabla v(x)|^q \,dx
		\ge C \int_{D_0} \left[ |\nabla_{r,z} v (r,z)|^q + r^{-q} |v^r (r,z)|^q \right] r \,drdz,
\end{align*}
and hence, $\omega^{\theta}$ and $\Omega$ satisfy
\begin{align}
\label{Diri_o}
	\int_{D_0} |\omega^{\theta}(r,z)|^q r \,drdz < \infty,\quad
	\int_{D_0} r^{q+1} | \Omega(r,z) |^q \,drdz < \infty,
\end{align}
respectively.
\par
\vspace{2mm}
From the above bound of $\Omega$ and the equation \eqref{eq_omega},
we prove the following estimate.
\begin{lemma}\label{lem_en_est}
Suppose the assumptions of Theorem \ref{thm_asymp}.
Let $r_1 > r_0$ and $D_1 = \{ (r,z) \in \re_+ \times \re ; \mbox{ }r \ge r_1 \}$.
Let $\alpha \le \min\{ q + 3, q+2-k \}$.
Then, we have
\begin{align*}
	\int_{D_1} r^{\alpha} |\Omega(r,z)|^{q-2} |\nabla \Omega(r,z)|^2 \,drdz
	\le C \int_{D_0} r^{q+1} |\Omega(r,z)|^q \,drdz, 
\end{align*}
where $C = C(q, \alpha, k, r_0, r_1)$.  
\end{lemma}
\begin{proof}
Let
$r_1, r_2$ be
$r_2 > r_1 > r_0$, and let
$\xi_1(r), \xi_2(r) \in C^{\infty}((0,\infty))$
be nonnegative functions satisfying
\begin{align*}
	\xi_1(r) = \begin{cases} 1 &(r \ge r_2),\\ 0 &(r_0< r \le r_1), \end{cases}\quad
	\xi_2(r) = \begin{cases} 1 &(0\le r \le 1/2),\\ 0 &(r \ge 1). \end{cases}
\end{align*}
For $R > 0$, we define a test function
\begin{align*}
	\eta_R(r, z) = \xi_1(r) \xi_2 \left( \frac{r}{R} \right) \xi_2 \left( \frac{|z|}{R} \right).
\end{align*}
Then, we see that
\begin{align}
\label{est_eta}
	|\nabla_{r,z} \eta_R(r,z)| \le C (|\xi_1'(r)| + R^{-1}), \quad
	|\Delta_{r,z} \eta_R(r,z)| \le C (|\xi_1''(r)| + R^{-1}|\xi_1'(r)| + R^{-2} ).
\end{align}

Let $h= h(\Omega)$ be a $C^2$ function determined later.
We start with the following identity (see also \cite[p.385]{GiWe78}):
\begin{align*}
	&\diver_{r,z} \left[ (r^{\alpha} \eta_R) \nabla_{r,z}h(\Omega)
					- h(\Omega) \nabla_{r,z} (r^{\alpha} \eta_R )
					- (r^{\alpha} \eta_R) h(\Omega) v \right] \\
	&= r^{\alpha} \eta_R h''(\Omega) |\nabla_{r,z} \Omega|^2
		- h(\Omega) \left[ \Delta_{r,z} (r^{\alpha} \eta_R)
						+ v \cdot \nabla_{r,z} (r^{\alpha} \eta_R ) \right] \\
	&\quad + r^{\alpha} \eta_R h'(\Omega) \left[ \Delta_{r,z} \Omega - v \cdot \nabla_{r,z} \Omega \right]
		- r^{\alpha} \eta_R h(\Omega) \diver_{r,z} v.
\end{align*}
By the equation \eqref{eq_omega}, we have
\begin{align*}
	& r^{\alpha} \eta_R h'(\Omega) \left[ \Delta_{r,z} \Omega - v \cdot \nabla_{r,z} \Omega \right] \\
	&= r^{\alpha} \eta_R h'(\Omega) \left( - \frac{3}{r} \partial_r \Omega \right) \\
	&= - 3 \partial_r \left( r^{\alpha-1} \eta_R h(\Omega) \right)
		+ 3 \partial_r \left( r^{\alpha-1} \eta_R \right) h(\Omega).
\end{align*}
Also, from the third line of \eqref{snsaxi}, we obtain
$- r^{\alpha} \eta_R h(\Omega) \diver_{r,z} v = r^{\alpha-1} \eta_R h(\Omega) v^r$.
These observations with a straightforward computation lead to
\begin{align*}
	&\diver_{r,z} \left[ (r^{\alpha} \eta_R) \nabla_{r,z}h(\Omega)
					- h(\Omega) \nabla_{r,z} (r^{\alpha} \eta_R )
					- (r^{\alpha} \eta_R) h(\Omega) v \right]
					+ 3 \partial_r (r^{\alpha-1} \eta_R h(\Omega)) \\
	&= r^{\alpha} \eta_R h''(\Omega) |\nabla_{r,z} \Omega|^2
		- h(\Omega) \left[ \Delta_{r,z} (r^{\alpha} \eta_R)
						+ v \cdot \nabla_{r,z} (r^{\alpha} \eta_R ) \right] \\
	&\quad
		+ 3 \partial_r (r^{\alpha-1}\eta_R ) h(\Omega)
		+ r^{\alpha-1} \eta_R h(\Omega) v^r\\ 
	&= r^{\alpha} \eta_R h''(\Omega) |\nabla_{r,z} \Omega|^2
		- h(\Omega) \left[ r^{\alpha} \Delta_{r,z} \eta_R + (2\alpha-3) r^{\alpha-1} \partial_r \eta_R \right] \\
	&\quad - h(\Omega) \left[ r^{\alpha} v \nabla_{r,z} \eta_R + (\alpha-1) r^{\alpha-1} v^r \eta_R \right]
		-(\alpha-3)(\alpha-1) h(\Omega) r^{\alpha-2} \eta_R.
\end{align*}
Taking $h(\Omega) = |\Omega|^q$ 
and integrating the above identity over
$D_0$,
we deduce
\begin{align*}
	&q(q-1) \int_{D_0} r^{\alpha} \eta_R |\Omega(r,z)|^{q-2} |\nabla_{r,z} \Omega(r,z) |^2 \,drdz \\
	&= \int_{D_0} |\Omega(r,z)|^q
		\left[ r^{\alpha} \Delta_{r,z} \eta_R + (2\alpha+3) r^{\alpha-1} \partial_r \eta_R \right] \, drdz \\
	&\quad + \int_{D_0}  |\Omega(r,z)|^q
		\left[ r^{\alpha} v \nabla_{r,z} \eta_R + (\alpha-1) r^{\alpha-1} v^r \eta_R \right] \,drdz \\
	&\quad + (\alpha+3)(\alpha-1) \int_{D_0} |\Omega(r,z)|^q r^{\alpha-2} \eta_R \,drdz.
\end{align*}
Applying \eqref{est_eta} and the assumption \eqref{ass_v},
we have by the property of the support of $\eta_R$ and its derivatives that  
\begin{align*}%
	&\int_{D_0} r^{\alpha} \eta_R |\Omega(r,z)|^{q-2} |\nabla_{r,z} \Omega(r,z) |^2 \,drdz \\
	&\le C \sum_{l=0}^2 R^{-l} \int_{r_0<r<R}\int_{-\infty}^\infty  r^{\alpha-2 + l} |\Omega(r,z)|^q \,drdz \\
	&\quad + C R^{-1} \int_{r_0< r < R}\int_{-\infty}^\infty  r^{k+\alpha} |\Omega(r,z)|^q \,drdz 
		+ C \int_{D_0} r^{k+\alpha -1} |\Omega(r,z)|^q \,drdz \\
	&\quad + C \int_{r_1<r<r_2}\int_{-\infty}^\infty  |\Omega(r,z)|^q \,drdz \\
	&\le C \int_{D_0} r^{\alpha-2} |\Omega(r,z)|^q \,drdz 
		+ C \int_{D_0} r^{k+\alpha-1} |\Omega(r,z)|^q \,drdz.
\end{align*}%
From the assumption $\alpha \le \min\{ q + 3, q+2-k \}$,
we obtain
\begin{align*}%
	&\int_{D_0} r^{\alpha} \eta_R |\Omega(r,z)|^{q-2} |\nabla_{r,z} \Omega(r,z) |^2 \,drdz 
	\le C \int_{D_0} r^{q+1} |\Omega(r,z)|^q \,drdz.
\end{align*}%
Finally, since the left-hand side is bounded from below by
\begin{align*}%
	\int_{D_1} r^{\alpha} |\Omega(r,z)|^{q-2} |\nabla_{r,z} \Omega(r,z) |^2 \,drdz, 
\end{align*}%
we conclude that 
\begin{align*}%
	\int_{D_1} r^{\alpha} |\Omega(r,z)|^{q-2} |\nabla_{r,z} \Omega(r,z) |^2 \,drdz
	\le C \int_{D_0} r^{q+1} |\Omega(r,z)|^q \,drdz.
\end{align*}%
This completes the proof of Lemma \ref{lem_en_est}.
\end{proof}

\subsection{Pointwise behavior via maximum principle}

The condition \eqref{Diri_o} and Lemma \ref{lem_en_est} give the boundedness
\begin{align*}%
	\int_{D_0} r^{q+1} |\Omega(r,z)|^q \,drdz < \infty,\quad
	\int_{D_1} r^{\alpha} |\Omega(r,z)|^{q-2} |\nabla \Omega(r,z)|^2 \,drdz < \infty
\end{align*}%
with $\alpha \le \min\{q+3, q+2-k\}$.  
The following proposition shows that
the above bounds with the maximum principle yield a pointwise behavior of $\Omega$
as $r \to \infty$.

\begin{proposition}\label{prop_pt}
Let $r_1 > 0$, $D_1 = \{ (r,z) \in \re^+ \times \re ;\mbox{ } r > r_1 \}$,
and let
$f = f(r,z) \in C^{\infty}(D_1)$ satisfy
\begin{align}%
\label{ass_f}
	\int_{D_1} r^{\alpha} | f (r,z)|^{q-2} |\nabla_{r,z} f(r,z) |^2 \,drdz
	+ \int_{D_1} r^{q+1} | f(r,z)|^q \,drdz < \infty
\end{align}%
with some
$q \in [2,\infty)$ and $\alpha \le q + 3$.
Moreover, we assume that $f$ satisfies the maximum principle, that is,
for each bounded domain $D \subset D_1$,
the function $f|_{\bar D}$ does not attain its maximum or minimum value in 
the interior of $D$.
Then, we have
\begin{align*}
	\lim_{r\to \infty} r^{\frac{\alpha+q+3}{2q}} \sup_{z \in \re} |f(r,z)| = 0.
\end{align*}
\end{proposition}
In order to prove this proposition,
we first show that the bounds \eqref{ass_f} leads to
a pointwise behavior along with a certain sequence
$\{ r_n \}_{n=1}^{\infty}$ satisfying $\lim_{n\to \infty} r_n = \infty$.

\begin{lemma}\label{lem_pt_est}
Let $r_1 > 0$, $D_1 = \{ (r,z) \in \re^+ \times \re ;\mbox{ }r > r_1 \}$,
and let
$f = f(r,z) \in C^{\infty}(D_1)$ satisfy the condition \eqref{ass_f}
with some
$q \in [2,\infty)$ and $\alpha \le q + 3$.
Then, there exists a sequence
$\{ r_n \}_{n=1}^{\infty}$ satisfying $\lim_{n\to \infty} r_n = \infty$
such that
\begin{align*}%
	\lim_{n\to \infty} r_n^{\frac{\alpha+q+3}{2q}} \sup_{z \in \re} | f(r_n, z) | = 0.
\end{align*}%
\end{lemma}
\begin{proof}
Let $n \in \mathbb{N}$ satisfy
$2^n > r_1$.
By the assumption and the Schwarz inequality, we have
\begin{align*}%
	&\int_{r > 2^n} \int_{-\infty}^\infty | f(r,z) |^{q-2} \left(
							r^{q+1} |f(r,z)|^2 + r^{\frac{\alpha+q+1}{2}} | f(r,z) | | \nabla_{r,z} f(r,z)|
							 \right) \,drdz \\
	&\le C  \int_{r > 2^n}\int_{-\infty}^\infty  | f(r,z) |^{q-2} \left(
							r^{q+1} |f(r,z)|^2 + r^{\alpha} |\nabla_{r,z} f(r,z) |^2
							\right) \,drdz \\
	&< \infty.
\end{align*}%
Here, we note that the Lebesgue dominated convergence theorem shows
\begin{align}%
\label{dom_con}
	\lim_{n\to \infty} \int_{r > 2^n}\int_{-\infty}^\infty  | f(r,z) |^{q-2}\left(
							r^{q+1} |f(r,z)|^2 + r^{\frac{\alpha+q+1}{2}} | f(r,z) | | \nabla_{r,z} f(r,z)|
							 \right) \,drdz = 0.
\end{align}%

The mean value theorem for integration implies that
there exists $r_n \in [2^n, 2^{n+1}]$ such that
\begin{align*}%
	&\int_{-\infty}^{\infty}
		|f(r_n,z)|^{q-2} \left(
				r_n^{q+2} |f(r_n,z)|^2 + r_n^{\frac{\alpha+q+3}{2}} |f(r_n,z)| | \nabla_{r,z}f(r_n,z)|
					\right) \,dz \\
	&= \frac{1}{\log 2} \int_{2^n}^{2^{n+1}}
			\frac{dr}{r} \int_{-\infty}^{\infty}
				|f(r,z)|^{q-2} \left(
						r^{q+2} |f(r,z)|^2 + r^{\frac{\alpha + q + 3}{2}} |f(r,z)| | \nabla_{r,z}f(r,z)|
					\right) \,dz \\
	&\le C \int_{r > 2^n}\int_{-\infty}^\infty  | f(r,z) |^{q-2} \left(
							r^{q+1} |f(r,z)|^2 + r^{\frac{\alpha+q+1}{2}} | f(r,z) | | \nabla_{r,z} f(r,z)|
							 \right) \,drdz.
\end{align*}%
We denote the left-hand side by $J_n$:
\begin{align*}%
	J_n := \int_{-\infty}^{\infty}
		|f(r_n,z)|^{q-2} \left(
				r_n^{q+2} |f(r_n,z)|^2 + r_n^{\frac{\alpha+q+3}{2}} |f(r_n,z)| | \nabla_{r,z}f(r_n,z)|
					\right) \,dz.
\end{align*}%
From the estimate above and \eqref{dom_con}, we obtain
\begin{align}%
\label{lim_rn}
	\lim_{n\to \infty} J_n = 0.
\end{align}%
On the other hand, by the fundamental theorem of calculus, we see that
for any $z_1, z_2 \in \re$,
\begin{align*}%
	&r_n^{\frac{\alpha+q+1}{2}} |f(r_n, z_1)|^q - r_n^{\frac{\alpha+q+1}{2}} |f(r_n,z_2)|^q \\
	& = r_n^{\frac{\alpha+q+1}{2}}
		\int_{z_2}^{z_1}  \partial_z \left( |f(r_n,z)|^q \right) \,dz \\
	&\le C r_n^{\frac{\alpha+q+1}{2}}
		\int_{-\infty}^{\infty} |f(r_n,z)|^{q-1} |\nabla_{r,z}f(r_n,z)| \,dz.
\end{align*}%
Integrating it over $[-r_n, r_n]$ with respect to $z_2$, we have
\begin{align*}%
	r_n^{\frac{\alpha+q+3}{2}} |f(r_n, z_1)|^q 
	&\le C  r_n^{\frac{\alpha+q+1}{2}} \int_{-r_n}^{r_n} | f(r_n,z_2) |^q \,dz_2 \\
	&\quad + C r_n^{\frac{\alpha+q+3}{2}} \int_{-\infty}^{\infty} |f(r_n,z)|^{q-1} |\nabla_{r,z}f(r_n,z)| \,dz \\
	&\le C J_n,
\end{align*}%
where we have used that
$\frac{\alpha+q+1}{2} \le q+2$,
which follows from the assumption
$\alpha \le q + 3$.
Since $z_1 \in \re$ is arbitrary, we conclude from \eqref{lim_rn} that 
\begin{align*}%
	\lim_{n\to \infty} r_n^{\frac{\alpha+q+3}{2}} \sup_{z \in \re}|f(r_n, z)|^q = 0,
\end{align*}%
and the proof of Lemma \ref {lem_pt_est} is now complete. 
\end{proof}

\begin{proof}[Proof of Proposition \ref{prop_pt}]
Let $\{r_n\}_{n=1}^{\infty}$ be the sequence given by Lemma \ref{lem_pt_est}.
Then, we have
\begin{align*}%
	\lim_{n\to \infty} r_n^{\frac{\alpha+q+3}{2q}} \sup_{z \in \re}|f(r_n, z)| = 0.
\end{align*}%
To prove Proposition \ref{prop_pt},
we apply the maximum principle to obtain the pointwise estimate
for general $r$ in the interval $[r_n, r_{n+1}]$.
However, before doing it,
we also need to control the pointwise behavior for $z$-direction.
Therefore, we claim that there exists a sequence
$\{ z_m \}_{m \in \mathbb{Z}}$
satisfying
$\lim_{m\to \pm \infty} z_m = \pm \infty$
such that
\begin{align}%
\label{pt_est_z}
	\lim_{m\to \pm \infty} |z_m| \sup_{r_n \le r \le r_{n+2}} r^{\frac{\alpha+q+1}{2}} | f(r, z_m) |^q = 0.
\end{align}%
The reason why we take the interval
$[r_n, r_{n+2}]$
instead of
$[r_n, r_{n+1}]$
is that
the length of the former interval $[r_n, r_{n+2}]$ has the bound 
from both above and below such as 
$r/4 \le r_{n+2} - r_n \le 8 r$
for all $r \in [r_n, r_{n+2}]$
(note that $r_n \in [2^n, 2^{n+1}]$).

Let us prove \eqref{pt_est_z}.
We fix $n \in \mathbb{N}$, namely, fix the interval $[r_n, r_{n+2}]$.
In the same way as in  the proof of Lemma \ref{lem_pt_est},
we see that there exists a sequence
$\{ z_m \}_{m \in \mathbb{Z}}$
satisfying
$z_{\pm l} \in [\pm 2^l, \pm 2^{l+1}]$
for each
$l \in \mathbb{N}$
and
\begin{align*}
	&\int_{r_n}^{r_{n+2}} |z_{\pm l}| |f(r,z_{\pm l})|^{q-2}
		\left( r^{q+1} |f(r,z_{\pm l})|^2 + r^{\frac{\alpha+q+1}{2}} |f(r,z_{\pm l})| | \nabla_{r,z}f(r,z_{\pm l})| \right)
		\,dr \\
	&\le C \left| \int_{\pm 2^l}^{\pm 2^{l+1}} \int_{r_n}^{r_{n+2}}
		|f(r,z)|^{q-2} \left( r^{q+1} |f(r,z)|^2 + r^{\frac{\alpha+q+1}{2}} |f(r,z)| | \nabla_{r,z}f(r,z)| \right)
		\, dz dr \right|.
\end{align*}
We denote the left-hand side by $K_{\pm l}$:
\begin{align*}
	K_{\pm l} := \int_{r_n}^{r_{n+2}} |z_{\pm l}| |f(r,z_{\pm l})|^{q-2}
		\left( r^{q+1} |f(r,z_{\pm l})|^2 + r^{\frac{\alpha+q+1}{2}} |f(r,z_{\pm l})| | \nabla_{r,z}f(r,z_{\pm l})| \right)
		\,dr.
\end{align*}
We note that the above estimate and the Lebesgue dominated convergence theorem
yield 
\begin{align}
\label{lim_km}
	\lim_{l\to \infty} K_{\pm l} = 0.
\end{align}
In what follows, for simplicity we abbreviate in such a way that
$m = \pm l$ and $K_m = K_{\pm l}$.
For every $r, \tilde{r} \in [r_n, r_{n+2}]$,
by the fundamental theorem of calculus, we deduce
\begin{align*}
	&|z_m| r^{\frac{\alpha+q-1}{2}} |f(r,z_m)|^q - 
	|z_m| \tilde{r}^{\frac{\alpha+q-1}{2}} |f(\tilde{r},z_m)|^q \\
	&= |z_m| \int_{\tilde{r}}^{r} \partial_{\rho} \left(
					\rho^{\frac{\alpha+q-1}{2}} |f(\rho,z_m)|^q \right) \,d\rho \\
	&\le C |z_m| \int_{r_n}^{r_{n+2}}
		\left[ \rho^{\frac{\alpha+q-3}{2}} |f(\rho,z_m)|^q
			+ \rho^{\frac{\alpha+q-1}{2}} |f(\rho,z_m)|^{q-1} |\nabla_{r,z} f(\rho,z_m)|
			\right] \,d\rho.
\end{align*}
Since $\alpha \le q + 3$, we have $\frac{\alpha + q-1}{2} \le q+1$, and hence 
integration over $[r_n, r_{n+2}]$ with respect to $\tilde{r}$ of 
both sides of the above estimate yields 
\begin{align*}
	&(r_{n+2}-r_n) |z_m| r^{\frac{\alpha+q-1}{2}} |f(r,z_m)|^q \\
	&\le |z_m| \int_{r_n}^{r_{n+2}} \tilde{r}^{\frac{\alpha+q-1}{2}} |f( \tilde{r}, z_m )|^q \,d\tilde{r} \\
	&\quad + (r_{n+2}-r_n) |z_m|
		\int_{r_n}^{r_{n+2}} \left[
			\rho^{\frac{\alpha+q-3}{2}} |f(\rho, z_m)|^q
				+ \rho^{\frac{\alpha+q-1}{2}} |f(\rho,z_m)|^{q-1} |\nabla_{r,z} f(\rho,z) |
				\right] \,d \rho \\				
	&\le K_m.
\end{align*}
Combining it with
$r_{n+2}-r_n \sim r$
and \eqref{lim_km}, we obtain the claim \eqref{pt_est_z}.

Finally, we apply the maximum principle to complete the proof of Proposition \ref{prop_pt}.
Let $r > r_1$ be sufficiently large and take $n \in \mathbb{N}$ so that
$r \in [r_n, r_{n+2}]$ with $\{r_n\}_{n=1}^\infty$ satisfying (\ref{lim_rn}).
We also take
$\{ z_m \}_{m \in \mathbb{Z}}$ so that \eqref{pt_est_z} holds.
Let
$D_{n,m} = \{ (r,z) \in \re^+ \times \re ;\mbox{ } r \in [r_n, r_{n+2}] \times [z_{-m}, z_m] \}$.
Then, by the maximum principle, for every $(r,z) \in D_{n,m}$, we have
\begin{align*}
	& r^{\frac{\alpha+q+3}{2q}} |f(r,z)| \\
	&	\le  \max  \left\{ r_{n+2}^{\frac{\alpha+q+3}{2q}} \sup_{w\in [z_{-m}, z_m]} 
		|f(r_n, w)|,\ 
			r_{n+2}^{\frac{\alpha+q+3}{2q}} \sup_{w\in [z_{-m}, z_m]} |f(r_{n+2}, w)|, \right. \\
				&\left.
			r^{\frac{\alpha+q+3}{2q}} \sup_{\rho \in [r_n, r_{n+2}]} |f(\rho ,z_m)|,\ 
			r^{\frac{\alpha+q+3}{2q}} \sup_{\rho\in [r_n, r_{n+2}]} |f(\rho,z_{-m})| \right\}\\
	& 	\le  \max  \left\{ 8^{\frac{\alpha+q+3}{2q}}r_{n}^{\frac{\alpha+q+3}{2q}} \sup_{w\in {\mathbb R}} 
		|f(r_n, w)|,\ 
			r_{n+2}^{\frac{\alpha+q+3}{2q}} \sup_{w\in {\mathbb R}} |f(r_{n+2}, w)|, \right. \\
				&\left.
			8^{\frac{\alpha+q+1}{2q}}r_{n+2}^{\frac{\alpha+q+2}{2q}} 
			\sup_{\rho \in [r_n, r_{n+2}]}\rho^{\frac{\alpha+q+1}{2q}} |f(\rho ,z_m)|,\ 
			8^{\frac{\alpha+q+1}{2q}}r_{n+2}^{\frac{\alpha+q+2}{2q}} 
			\sup_{\rho \in [r_n, r_{n+2}]}\rho^{\frac{\alpha+q+1}{2q}} 
|f(\rho ,z_{-m})| \right\}.
\end{align*}
Here we have used that $r_{n+2} \le 8r_n$.  
Letting $m \to \infty$ in the above, we have by \eqref{pt_est_z} that 
\begin{align*}
	r^{\frac{\alpha+q+3}{2q}} \sup_{z \in \re} |f(r,z)| 
	\le  \max\left\{
			8^{\frac{\alpha+q+3}{2q}}r_n^{\frac{\alpha+q+3}{2q}} 
			\sup_{w\in \re} |f(r_n, w)|,\ 
			r_{n+2}^{\frac{\alpha+q+3}{2q}} \sup_{w\in \re} |f(r_{n+2}, w)| \right\}
\end{align*}
holds for $r \in [r_n, r_{n+2}]$.
Consequently, taking the limit
$r \to \infty$ (with $n\to \infty$), we obtain from  Lemma \ref{lem_pt_est} that 
\begin{align*}
	\lim_{r\to \infty} r^{\frac{\alpha+q+3}{2q}} \sup_{z \in \re} |f(r,z)| = 0,
\end{align*}
which yields the desired estimate. 
This completes the proof of Proposition \ref{prop_pt}.
\end{proof}

\subsection{Proof of Theorem \ref{thm_asymp}}
By the assumptions on Theorem \ref{thm_asymp},
we apply Lemma \ref{lem_en_est} and Proposition \ref{prop_pt}
with $\alpha = \min\{ q + 3, q+2-k \}$
to obtain
\begin{align*}
	\lim_{r\to \infty} r^{ 1+\frac{3}{q} - \frac{1}{2q} \max\{0,1+k\} } \sup_{z \in \re} |\Omega(r,z)| = 0.
\end{align*}
Concerning the estimate of  $\omega^{\theta}$, we have by the relation 
$\omega^{\theta} = r \Omega$ that 
\begin{align*}
	\lim_{r\to \infty} r^{\frac{3}{q} - \frac{1}{2q} \max\{0,1+k\} } \sup_{z \in \re} |\omega^{\theta} (r,z)| = 0.
\end{align*}
This completes the proof of Threorem \ref{thm_asymp}.
\subsection{Proof of Corollary \ref{cor_liouville}} 
Since $k \le 2q + 5$, we see that $1 + \frac3q -\frac{1}{2q}\max\{0, 1+k\} \ge 0$, 
and hence it follows from (\ref{asym_omega}) that 
$$
\lim_{r\to\infty}\sup_{z\in {\mathbb R}}|\Omega(r, z)| =0.  
$$
Since $\lim_{|z|\to\infty}|\Omega(r, z| = 0$ for each fixed $r > 0$, 
we obtain from the maximum principle that 
$\Omega \equiv 0$ on $\re^3$. Since $\omega^\theta \in C(\re^3)$, we conclude that 
$\omega^\theta \equiv 0$ on $\re^3$.   
This proves Corollary \ref{cor_liouville}.

\section*{acknowledgement}
This work was supported by JSPS Grant-in-Aid for Scientific Research(S)
Grant Number JP16H06339.

\end{document}